\documentclass[12pt]{article}
\usepackage{mathrsfs}
\usepackage{}
\usepackage{amscd}
\usepackage{amsfonts}
\usepackage{amsmath}
\usepackage{amsthm}
\usepackage{amssymb}
\usepackage{times}
\usepackage[dvips, hyperindex]{hyperref}
\usepackage{enumerate}
\usepackage{amsfonts, latexsym}
\usepackage{stmaryrd}

\numberwithin{equation}{section}
\newtheorem{thm}[equation]{Theorem}

\newtheorem{cor}[equation]{Corollary}

\theoremstyle{definition}
\newtheorem{example}[equation]{Example}

\DeclareMathOperator{\Mod}{mod}

\begin{document}

\title{Congruences for the number of partitions and bipartitions with distinct even parts}
\author{
        Haobo Dai\footnote{Department of Mathematics,
University of Shanghai Jiao Tong University, Shanghai, 200240,
China; e-mail: {\tt dedekindbest@hotmail.com}}
}

\date{}
\maketitle

\begin{abstract}
Let $ped(n)$ denote the number of partitions of $n$ wherein even parts are distinct (and odd parts are unrestricted). We show infinite families of congruences for $ped(n)$ modulo $8$. We also examine the behavior of $ped_{-2}(n)$ modulo $8$ in detail where $ped_{-2}(n)$ denotes the number of bipartitions of $n$ with even parts distinct. As a result, we find infinite families of congruences for $ped_{2}(n)$ modulo $8$.
\\


\end{abstract}

\section{Introduction}

Let $ped(n)$ denote the number of partitions of $n$ wherein even parts are distinct (and odd parts are unrestricted). The generating function for $ped(n)$ (\cite{AHS}) is
$$\sum_{n=0}^{\infty}ped(n)q^n:=\frac{(q^4;q^4)_{\infty}}{(q;q)_{\infty}}=\prod_{m=1}^{\infty}\frac{(1-q^{4m})}{(1-q^m)}. \eqno(1.1)$$
Note that by (1.1), the number of partitions of $n$ wherein even parts are distinct equals the number of partitions of $n$ with no parts divisible by $4$, i.e., the $4$-regular partitions (see \cite{AHS} and references therein).
The arithmetic properties were studied by Andrews, Hirshhorn and Sellers \cite{AHS} and Chen \cite{Ch}. For example,  in \cite {AHS}, Andrews, et al., proved that for all $n\geq 0$,
$$ped(9n+4)\equiv 0 \quad (\Mod 4)  \eqno(1.2)$$
and
$$ped(9n+7)\equiv 0 \quad (\Mod 4).\eqno(1.3)$$
Suppose that $r$ is an integer such that $1\leq r <8p$, $rp\equiv 1(\Mod 8)$, and $(r, p)=1$. By using modular forms, Chen \cite{Ch} showed that if $c(p)\equiv 0(\Mod 4)$, then, for all $n\geq 0$, $\alpha\geq 1$,
$$ped\left(p^{2\alpha}n+\frac{rp^{2\alpha-1}-1}{8}\right)\equiv 0  \quad (\Mod 4) \eqno(1.4)$$
where $c(p)$ is the $p$-th coefficient of $\frac{\eta^4(16z)}{\eta(8z)\eta(32z)}:=\sum_{n=1}^{\infty}c(n)q^n$. Note that in \cite{Ch}, Chen didn't show the coefficients of $c(p)$ explicitly. Note also that in a beautiful paper \cite{Ch1}, Chen studied arithmetic properties for the number of $k$-tuple partitions with even parts distinct modulo $2$ for any positive integer $k$ by using Hecke nilpotency.
Berkovich and Patane \cite{B-K} calculated $c(n)$ explicitly. In particular, they showed that $c(p)=0$ if and only if $p=2$, $p\equiv 5 (\Mod 8)$ and $p\equiv 3 (\Mod 4)$.

 As a direct application of Chen's, Berkovich and Patane's theorems, we have the following.

\begin{cor}
Let $p$ be a prime which is congruent  to 5 modulo $8$ or congruent to $3$ modulo $4$ and suppose that $r$ is an integer such that $1\leq r <8p$, $rp\equiv 1(\Mod 8)$, and $(r, p)=1$, then
$$ped\left(p^{2\alpha}n+\frac{rp^{2\alpha-1}-1}{8}\right)\equiv 0  \quad (\Mod 4)$$
for all $n\geq 0$ and $\alpha\geq 1$.
\end{cor}

Ono and Penniston \cite{O-P} showed an explicit formula for $Q(n)$ modulo $8$ by using the arithmetic of the ring of $\mathbb{Z}[\sqrt{-6}]$ where $Q(n)$ denotes the number of partitions of an integer $n$ into distinct parts. We are unable to explicitly determine $ped(n)$ modulo $8$. But we can prove infinitely families congruences for $ped(n)$ modulo $8$. Our first main result is the following.

\begin{thm}
Let $p$ be a prime which is congruent to $7 (\Mod 8)$. Suppose that $r$ is an integer such that $1\leq r <8p$, $rp\equiv 1(\Mod 8)$, and $(r, p)=1$, then for all $n\geq 0$, $\alpha\geq 0$, we have
$$ped\left(p^{2\alpha+2}n+\frac{rp^{2\alpha-1}+1}{8}\right)\equiv 0  \quad (\Mod 8).$$

\end{thm}

\begin{example}
For all $n\geq 0$, $\alpha\geq 0$,
$$ped\left(7^{2\alpha+2}n+\frac{r\times 7^{2\alpha+1}-1}{8}\right)\equiv 0  \quad (\Mod 8),$$
for $r=15,23,31,39$ and $47$.
\end{example}

Let $ped_{-2}(n)$ be the number of bipartitions of $n$ with even parts distinct. The generating function of $ped_{-2}(n)$ \cite{Lin} is
$$\sum_{n=0}^{\infty}ped_{-2}(n)q^n:=\frac{(q^4;q^4)^2_{\infty}}{(q;q)^2_{\infty}}=\prod_{m=1}^{\infty}\frac{(1-q^{4m})^2}{(1-q^m)^2}. \eqno(1.5)$$
Recently, in \cite{Lin}, Lin investigated arithmetic properties of $ped_{-2}(n)$. In particular, he showed following theorems:

\begin{thm}(\cite{Lin})
For $\alpha\geq 0$ and any $n\geq 0$, we have
$$ped_{-2}(n)\left(3^{2\alpha+2}n+\frac{11\times 3^{2\alpha+1}-1}{4}\right)\equiv 0 \quad (\Mod 3), $$
$$ped_{-2}(n)\left(3^{2\alpha+3}n+\frac{5\times 3^{2\alpha+2}-1}{4}\right)\equiv 0 \quad (\Mod 3). $$
\end{thm}

\begin{thm}[\cite{Lin}]
$ped_{-2}(n)$ is even unless $n$ is of the form $k(k+1)$ for some $k\geq 0$. Furthermore, $ped_{-2}(n)$ is a multiple of $4$ if $n$ is not the sum of two triangular numbers.
\end{thm}

As a corollary of Theorem 1.3 and 1.4, Lin proved an infinite family of congruences for $ped_{-2}(n)$ modulo $12$:
$$ped_{-2}\left(3^{2\alpha+2}n+\frac{11\times 3^{2\alpha+1}-1}{4}\right)\equiv 0  \quad (\Mod 12),$$
for any integers $\alpha\geq 0$ and $n\geq 0$.

As in \cite{O-P}, our second main achievement is to examine $ped_{-2}(n)$ modulo $8$ in detail.

\begin{thm}
If $n$ is a non-negative integer, then let $N$ and $M$ be the unique positive integers for which
$$4n+1=N^2M$$
where $M$ is square-free. Then the following are true:
\begin{itemize}
  \item [(1)] If $M=1$, then $ped_{-2}(n)$ is odd.
  \item [(2)] If $M=p$, and ord$_p(4n+1)\equiv 1 (\Mod 4)$, then $ped_{-2}(n)\equiv 2 (\Mod 4)$.
  \item [(3)] If $M=p$, and ord$_p(4n+1)\equiv 3 (\Mod 8)$, then $ped_{-2}(n)\equiv 4 (\Mod 8)$.
  \item [(4)] If $M=p_1p_2$, where $p_1$ and $p_2$ are distinct primes, $p_i\equiv 1 (\Mod 4)$ and ord$_{p_i}(4n+1)\equiv 1 (\Mod 4)$ for $i=1,2$, then $ped_{-2}(n)\equiv 4(\Mod 8)$.
  \item [(5)] In all other cases we have that $ped_{-2}(n)\equiv 0(\Mod 8).$
\end{itemize}

\end{thm}

As a corollary to Theorem 1.5. we can show infinite families of congruences for $ped_{-2}(n)$ modulo $8$.

\begin{cor}
\begin{itemize}
  \item [(i)] For $p\equiv 3 (\Mod 4)$, and let $1\leq r< 4p$, $(r,p)=1$, $rp\equiv 1 (\Mod 4)$, then we have
  $$ped_{-2}\left(p^{2\alpha+2}n+\frac{r\times p^{2\alpha+1}-1}{4}\right)\equiv 0  \quad (\Mod 8), \eqno(1.6)$$
  for $\alpha\geq 0$ and any $n\geq 0$.
  \item [(ii)] For $p\equiv 1 (\Mod 4)$, and let $1\leq r< 4p$, $(r,p)=1$, $rp\equiv 1 (\Mod 4)$, then we have
  $$ped_{-2}\left(p^{8\alpha+8}n+\frac{r\times p^{8\alpha+7}-1}{4}\right)\equiv 0  \quad (\Mod 8),\eqno(1.7)$$
  for $\alpha\geq 0$ and any $n\geq 0$.
\end{itemize}

\end{cor}

\begin{proof}
Since the proof of (i) and (ii) are similar, we only give the proof of (ii). Note that
$$4\left(p^{8\alpha+8}n+\frac{r\times p^{8\alpha+7}-1}{4}\right)+1=p^{8\alpha+7}(4pn+r),$$
so for any $n$, then from Theorem 1.7, (1.7) is obvious.

\end{proof}

\begin{example}
For $r=7,11$, for all $n\geq 0$, $\alpha\geq 0$, we have
$$ped_{-2}\left(3^{2\alpha+2}n+\frac{r\times 3^{2\alpha+1}-1}{4}\right)\equiv 0  \quad (\Mod 8).$$
Combining Lin's Theorem 1.5, we have
$$ped_{-2}\left(3^{2\alpha+2}n+\frac{11\times 3^{2\alpha+1}-1}{4}\right)\equiv 0  \quad (\Mod 24).$$
\end{example}

In section 2, we prove Theorem 1.2 and  in section 3 we prove Theorem 1.7. Our method is elementary which is motivated by that of Mahlburg's \cite{Ma}. We only use the knowledge of binary quadratic forms.

\vspace{2ex} {\bf \noindent Acknowledgments.} \\
   We wish to thank the NSF of China (No.11071160) for its generous support. We would also like to thank the referee for his/her helpful comments.

\section{Proof of Theorem 1.2}
In this section, we prove Theorem 1.2. We need the following well-known identities \cite{On}:
\setcounter{equation}{0}
\begin{eqnarray}
\prod_{n=1}^{\infty}\frac{(1-q^n)^2}{(1-q^{2n})}&=&\sum_{n=-\infty}^{\infty}(-1)^nq^{n^2}=1+2\sum_{n=1}^{\infty}(-1)^nq^{n^2},\\
q\prod_{n=1}^{\infty}\frac{(1-q^{16n})^2}{(1-q^{8n})}&=&\sum_{n=0}^{\infty}q^{(2n+1)^2}.
\end{eqnarray}

\begin{proof}[Proof of Theorem 1.2]
Since
   $$\sum_{n=0}^{\infty}ped(n)q^n: =\prod_{n=1}^{\infty}\frac{(1-q^{4n})}{(1-q^n)}=\prod_{n=1}^{\infty}\frac{(1-q^{2n})^2}{(1-q^n)}\cdot \frac{(1-q^{4n})}{(1-q^{2n})^2} ,$$
so we have
\begin{eqnarray*}
   & &\sum_{n=0}^{\infty}ped(n)q^{8n+1}\\
   &=& q\prod_{n=1}^{\infty}\frac{(1-q^{16n})^2}{1-q^{8n}}\cdot \frac{(1-q^{32n})}{(1-q^{16n})^2}\\
   &=& \left(\sum_{n=0}^{\infty}q^{(2n+1)^2}\right)\cdot \frac{1}{1+2\sum_{n=1}^{\infty}(-1)^n q^{16n^2}} \quad  \quad (\text{by}\ (2.1),(2.2))\\
   &\equiv& \left(\sum_{n=0}^{\infty}q^{(2n+1)^2}\right) \left(1-2\sum_{n=1}^{\infty}(-1)^nq^{16n^2}+4\left(\sum_{n=1}^{\infty}(-1)^nq^{16n^2}\right)^2\right) (\Mod 8).
\end{eqnarray*}
From above, it is clear that $ped(n)$ is odd if and only if $8n+1$ is a square. Note that
\begin{eqnarray*}
   \left(\sum_{n=1}^{\infty}(-1)^nq^{16n^2}\right)^2
   &=&\left(\sum_{m_1,m_2=1}^{\infty}(-1)^{m_1+m_2}q^{16m_1^2+16m_2^2}\right)\\
   &=& 2\left(\sum_{m_1,m_2=1 \atop m_1< m_2}^{\infty}(-1)^{m_1+m_2}q^{16m_1^2+16m_2^2}\right)
    +\left(\sum_{n=1}^{\infty}q^{32n^2}\right).
                \end{eqnarray*}
   So we have
   \begin{eqnarray}
   \sum_{n=0}^{\infty}ped(n)q^{8n+1}
   &\equiv&\left(1-2\sum_{n=1}^{\infty}(-1)^nq^{16n^2}+4\left(\sum_{n=1}^{\infty}q^{32n^2}\right)\right) \\
     &&\times\left(\sum_{n=0}^{\infty}q^{(2n+1)^2}\right)\quad  (\Mod 8).
                \end{eqnarray}
 Note that if $8n+1=p^{2\alpha+1}(8pm+r)$ where $p\equiv 7 (\Mod 8)$, $(r,p)=1$, $rp\equiv 1 (\Mod 8)$, then $8n+1$ can't be represented by $x^2$, $x^2+y^2$ or $x^2+2y^2$. So for these $8n+1$,  from (2.3) and (2.4), it is easy to see that $ped(n)\equiv 0 (\Mod 8)$. This completes the proof of Theorem 1.2.
 \end{proof}

\section{Proof of Theorem 1.7}

In this section, we prove Theorem 1.7. The method is similar to that of Theorem 1.2.

\begin{proof}[Proof of Theorem 1.7]
Since
   $$\sum_{n=0}^{\infty}ped_{-2}(n)q^n: =\prod_{n=1}^{\infty}\frac{(1-q^{4n})^2}{(1-q^n)^2}=\prod_{n=1}^{\infty}\frac{(1-q^{4n})^2}{(1-q^{2n})}\cdot \frac{(1-q^{2n})}{(1-q^{n})^2} ,$$
so by (2.1) and (2.2), we have
\begin{eqnarray*}
\sum_{n=0}^{\infty}ped_{-2}(n)q^{4n+1}
   &=& q\prod_{n=1}^{\infty}\frac{(1-q^{16n})^2}{1-q^{8n}}\cdot \frac{(1-q^{8n})}{(1-q^{4n})^2}\\
   &=&\frac{1}{1+2\sum_{n=1}^{\infty}(-1)^n q^{4n^2}}\times  \left(\sum_{n=0}^{\infty}q^{(2n+1)^2}\right)\\
   &\equiv& \left(1-2\sum_{n=1}^{\infty}(-1)^nq^{4n^2}+4\left(\sum_{n=1}^{\infty}(-1)^nq^{4n^2}\right)^2\right)\\
    &&\times\left(\sum_{n=0}^{\infty}q^{(2n+1)^2}\right)\quad  (\Mod 8).
\end{eqnarray*}
Note that
\begin{eqnarray*}
\left(\sum_{n=1}^{\infty}(-1)^nq^{4n^2}\right)^2
   &=& \left(\sum_{m_1,m_2=1}^{\infty}(-1)^{m_1+m_2}q^{4m_1^2+4m_2^2}\right)\\
   &=& 2\left(\sum_{m_1,m_2=1 \atop m_1< m_2}^{\infty}(-1)^{m_1+m_2}q^{4m_1^2+4m_2^2}\right)+\sum_{n=1}^{\infty}q^{8n^2}\\
   &=&2\left(\sum_{m_1,m_2=1 \atop m_1< m_2}^{\infty}(-1)^{m_1+m_2}q^{4m_1^2+4m_2^2}\right)+\frac{1}{2}\left(\sum_{n=\infty}^{\infty}q^{8n^2}-1\right).
                \end{eqnarray*}

 Note also that
\begin{eqnarray*}
                              \sum_{n=1}^{\infty}(-1)^nq^{4n^2}
                               &=& \frac{1}{2}\left(\sum_{n=-\infty}^{\infty}(-1)^nq^{4n^2}-1\right)\\
                               &=& -\frac{1}{2}-\frac{1}{2}\left(\sum_{n=-\infty}^{\infty}q^{4n^2}\right)
                                 +\left(\sum_{n=-\infty}^{\infty}q^{16n^2}\right).
                \end{eqnarray*}
 So we have
 \begin{eqnarray}
 \sum_{n=0}^{\infty}ped_{-2}(n)q^{4n+1}&\equiv&
 \left(\sum_{n=-\infty}^{\infty}q^{4n^2}-2\sum_{n=-\infty}^{\infty}q^{16n^2}+2\sum_{n=\infty}^{\infty}q^{8n^2}\right)\nonumber\\
                                    &&   \times\frac{1}{2}\left(\sum_{n=\infty}^{\infty}q^{(2n+1)^2}\right) \nonumber\\
                                     &\equiv& \left(\frac{1}{2}\sum_{n=-\infty}^{\infty}q^{4n^2}+\sum_{n=\infty}^{\infty}q^{8n^2}-\sum_{n=-\infty}^{\infty}q^{16n^2}\right)\\
                                      & &\times \left(\sum_{n=\infty}^{\infty}q^{(2n+1)^2}\right)\quad  (\Mod 8).
\end{eqnarray}
If we put
   $$\left(\sum_{n=\infty}^{\infty}q^{(2n+1)^2}\right)\left(\sum_{n=\infty}^{\infty}q^{4n^2}\right):=\sum_{n=0}^{\infty}a(n)q^{4n+1},\eqno(3.3)$$
 $$\left(\sum_{n=\infty}^{\infty}q^{(2n+1)^2}\right)\left(\sum_{n=\infty}^{\infty}q^{8n^2}\right):=\sum_{n=0}^{\infty}b(n)q^{4n+1},\eqno(3.4)$$
 and
 $$\left(\sum_{n=\infty}^{\infty}q^{(2n+1)^2}\right)\left(\sum_{n=\infty}^{\infty}q^{16n^2}\right):=\sum_{n=0}^{\infty}c(n)q^{4n+1},\eqno(3.5)$$
then it is easy to see from (3.1), (3.2) that
\begin{itemize}
  \item [(1)] $ped_{-2}(n)$ is odd if and only if $\frac{a(n)}{2}$ is odd,
  \item [(2)] $ped_{-2}(n)\equiv 2 (\Mod 4)$ if and only if $\frac{a(n)}{2}+b(n)-c(n)\equiv 2 (\Mod 4)$,
  \item [(3)] $ped_{-2}(n)\equiv 4 (\Mod 8)$ if and only if $\frac{a(n)}{2}+b(n)-c(n)\equiv 4 (\Mod 8)$,
  \item [(4)] In all other cases, we have that $ped_{-2}(n)\equiv 0(\Mod 8)$.
\end{itemize}

We separate into two cases according to the parity of $n$.

(i) If $n=2m$ is even, then it is easy to see that
\begin{eqnarray*}
   &&\#\{(x,y)\in \mathbb{Z}\times \mathbb{Z}:8m+1=x^2+16y^2,  8m+1 \ \text{is not a square}\}\\
   &=&\#\{(x,y)\in \mathbb{Z}\times \mathbb{Z}:8m+1=x^2+4y^2,  8m+1 \ \text{is not a square}\}\\
   &=& \frac{1}{2}\#\{(x,y)\in \mathbb{Z}\times \mathbb{Z}:8m+1=x^2+y^2,  8m+1 \ \text{is not a square}\},
\end{eqnarray*}
and
\begin{eqnarray*}
   &&\#\{(x,y)\in \mathbb{Z}\times \mathbb{Z}:8m+1=x^2+8y^2,  8m+1 \ \text{is not a square}\}\\
   &=&\#\{(x,y)\in \mathbb{Z}\times \mathbb{Z}:8m+1=x^2+2y^2,  8m+1 \ \text{is not a square}\}
\end{eqnarray*}
Note that if $8m+1=p_1^{\alpha_1}\cdots p_k^{\alpha_k}q_1^{\beta_1}\cdots q_l^{\beta_l}r_1^{\gamma_1}\cdots r_u^{\gamma_u}s_1^{\delta_1}\cdots s_v^{\delta_v}$, where $p_i\equiv 1$ $(\Mod 8)$, $q_i\equiv 3 (\Mod 8)$, $r_i\equiv 5 (\Mod 8)$ and $s_i\equiv 7 (\Mod 8)$, then by using the decomposition of prime ideals in $\mathbb{Z}[i]$, we know that
   $8m+1=x^2+y^2$ has integral solutions if and only if $\beta_j,\gamma_t\equiv 0 (\Mod 2)$ for $1\leq j\leq l$ and $1\leq t\leq u$. If all $\beta_j,\delta_t$ are even, then it is easy to see that
   $$\#\{(x,y)\in \mathbb{Z}\times \mathbb{Z}:8m+1=x^2+y^2\}=4(\alpha_1+1)\cdots (\alpha_k+1)(\gamma_1+1)\cdots(\gamma_u+1).$$
   We obtain that $a(n)=2(\alpha_1+1)\cdots (\alpha_k+1)(\gamma_1+1)\cdots(\gamma_u+1)$ if all $\beta_j,\gamma_t$ are even. Now it is clear that $\frac{a(n)}{2}$ is odd if and only if $4n+1=8m+1$ is a square.

   Similarly, by using the decomposition of prime ideals in $\mathbb{Z}[\sqrt{-2}]$, we know that if all $\gamma_j,\delta_t\equiv 0(\Mod 2)$, then (there are only two roots of unities $\pm 1$ in $\mathbb{Z}[\sqrt{-2}]$)
   $$\#\{(x,y)\in \mathbb{Z}\times \mathbb{Z}:8m+1=x^2+2y^2\}=2(\alpha_1+1)\cdots (\alpha_k+1)(\beta_1+1)\cdots(\beta_l+1).$$
   From the above argument, we obtain the following results. Suppose all $\delta_i$ are even, then
   \begin{itemize}
     \item [(a)] If $\beta_i$ is odd for some $1\leq i\leq l$ and all $\gamma_j$ is even, then $\frac{a(n)}{2}+b(n)-c(n)\equiv b(n)\equiv 2(\alpha_1+1)\cdots (\alpha_k+1)(\beta_1+1)\cdots(\beta_l+1)\equiv 0(\Mod 8)$ (because there must be another $\beta_{i'}$ odd for $p_i\equiv 3 (\Mod 8)$ while $8m+1\equiv 1 (\Mod 8)$).
     \item [(b)] If $\gamma_i$ is odd for some $1\leq i\leq u$ and all $\beta_j$ is even, then $\frac{a(n)}{2}+b(n)-c(n)\equiv \frac{a(n)}{2}-c(n)\equiv -(\alpha_1+1)\cdots (\alpha_k+1)(\gamma_1+1)\cdots(\gamma_u+1)(\Mod 8)$.
     \item [(c)] If all $\beta_i$ and $\gamma_j$ are even, then  $\frac{a(n)}{2}+b(n)-c(n)\equiv -(\alpha_1+1)\cdots (\alpha_k+1)(\gamma_1+1)\cdots(\gamma_u+1)+2(\alpha_1+1)\cdots (\alpha_k+1)(\beta_1+1)\cdots(\beta_l+1) (\Mod 8)$.
      \item [(d)]  If $\beta_i$ and $\gamma_j$ are odd for some $i$ and $j$, then clearly $ped_{-2}(n)\equiv 0 (\Mod 8)$.
   \end{itemize}
From (a),(b),(c) and (d), it is not difficult to see that (for $n=2m$)
\begin{itemize}
  \item [(1)] if $4n+1$ is a square, then $ped_{-2}(n)$ is odd,
  \item [(2)] if $4n+1=pa^2$ where ord$_p(4n+1)\equiv 1 (\Mod 4)$, then $ped_{-2}(n)\equiv 2 (\Mod 4)$,
  \item [(3)] if $4n+1=p_1p_2a^2$ where $p_i\equiv1 (\Mod 4)$ and ord$_{p_i}(4n+1)\equiv 1 (\Mod 4)$ or $4n+1=p^3a^2$ and ord$_p(4n+1)\equiv 3(\Mod 8)$, then $ped_{-2}(n)\equiv 4 (\Mod 8)$
  \item [(4)] and for all other cases, $ped_{-2}(n)\equiv 0 (\Mod 8)$.
\end{itemize}

(ii) If $n=2m+1$, then $4n+1=8m+5$. Since $(2m_1+1)^2+8m_2^2\equiv(2m_1+1)^2+16m_2^2\equiv 1(\Mod 8)$, so $4n+1$ can not be represented by $(2m_1+1)^2+8m_2^2$ and $(2m_1+1)^2+16m_2^2$. Note also that
\begin{eqnarray*}
   &&\#\{(x,y)\in \mathbb{Z}\times \mathbb{Z}:8m+5=x^2+4y^2\}\\
   &=& \frac{1}{2}\#\{(x,y)\in \mathbb{Z}\times \mathbb{Z}:8m+5=x^2+y^2 \}.
\end{eqnarray*}
So if $8m+5=p_1^{\alpha_1}\cdots p_k^{\alpha_k}q_1^{\beta_1}\cdots q_l^{\beta_l}r_1^{\gamma_1}\cdots r_u^{\gamma_u}s_1^{\delta_1}\cdots s_v^{\delta_v}$, where $p_i\equiv 1 (\Mod 8)$, $q_i\equiv 3 (\Mod 8)$, $r_i\equiv 5 (\Mod 8)$ and $s_i\equiv 7 (\Mod 8)$, then
$$\frac{a(n)}{2}\equiv (\alpha_1+1)\cdots (\alpha_k+1)(\gamma_1+1)\cdots(\gamma_u+1) \quad (\Mod 8). \eqno(3.6)$$
Just like in (i), it is easy to find case by case by (3.6) for $ped_{n} (\Mod 8)$ for $n=2m+1$:
\begin{itemize}
  \item [(1)] if $4n+1=pa^2$ where ord$_p(4n+1)\equiv 1 (\Mod 4)$, then $ped_{-2}(n)\equiv 2 (\Mod 4)$,
  \item [(2)] if $4n+1=p_1p_2a^2$ and ord$_{p_i}(4n+1)\equiv 1 (\Mod 4)$ or $4n+1=p^3a^2$ and ord$_p(4n+1)\equiv 3(\Mod 8)$, then $ped_{-2}(n)\equiv 4 (\Mod 8)$
  \item [(3)] and for all other cases, $ped_{-2}(n)\equiv 0 (\Mod 8)$.
\end{itemize}

Now combining (i) and (ii), with a little thought, we complete the proof of Theorem 1.7.

\end{proof}

\end{document}